\documentclass[12pt,a4paper]{article}
\usepackage{amsmath}
\usepackage{tikz}
\usepackage{amssymb}
\usepackage{amsthm}
\usepackage{amsmath}
\usepackage{amssymb}
\usepackage{bbm}
\usepackage{url}
\usepackage{amscd}
\usepackage{mathrsfs}
\usepackage{graphicx}
\usepackage{xcolor}
\usepackage{float}
\usepackage[margin=1in]{geometry}\newtheorem{theorem}{Theorem}

\newtheorem{lemma}[theorem]{Lemma}

\newtheorem*{remark}{Remark}

\newcommand{\isep}{\!\mathrel{{.}{.}\!}\nobreak}

\newcommand{\ep}{\varepsilon}

\newcommand{\wt}{\mathrm{wt}}
\newcommand{\E}{\mathbb{E}}
\renewcommand{\a}{\mathbf{a}}
\renewcommand{\b}{\mathbf{b}}
\newcommand{\x}{\mathbf{x}}
\newcommand{\y}{\mathbf{y}}
\newcommand{\0}{\mathbf{0}}
\newcommand{\A}{\mathcal{A}}
\newcommand{\Aa}{\mathcal{A}}
\newcommand{\ct}{\mathrm{ct}_{\Aa}}
\newcommand{\cts}{\widetilde{\mathrm{ct}}_{\Aa}}

\newcommand{\bo}{\mathbbm{1}}

\newcommand{\gaa}{\gamma_{\Aa}(q)}

\newcommand{\mua}{\mu_{\Aa}(q,m,n,r)}
\newcommand{\siga}{\sigma_{\Aa}^2(q,m,n,r)}

\newcommand{\M}{\mathbf{M}}
\newcommand{\NN}{\mathbf{N}}
\newcommand{\Y}{\mathbf{Y}}
\newcommand{\X}{\mathbf{X}}

\newcommand{\N}{\NN}

\renewcommand{\S}{\mathcal{S}}
\newcommand{\F}{\mathbb{F}}
\newcommand{\Z}{\mathbb{Z}}
\renewcommand{\P}{\mathbb{P}}
\renewcommand{\l}{\ell}

\newcommand{\Var}{\mathrm{Var}}
\newcommand{\Fmn}{\mathbb{F}_q^{m \times n,r}}

\title{The central limit theorem for entries of random matrices with specific rank over finite fields}
\author{Chin Hei Chan\thanks{C. Chan is at the Dept. of Mathematics, Hong Kong University of Science and Technology, Clear Water Bay, Kowloon, Hong Kong (email: chchanam@connect.ust.hk). C. Chan would like to acknowledge the financial support provided by the math department of HKUST so that this work is possible. }, Maosheng Xiong\thanks{M. Xiong is at the Dept. of Mathematics, Hong Kong University of Science and Technology, Clear Water Bay, Kowloon, Hong Kong (email: mamsxiong@ust.hk). The research of M. Xiong was supported by RGC grant number 16307524 from Hong Kong.}}

\date{}
\begin{document}
\maketitle

\begin{abstract}
Let $\F_q$ be the finite field of order $q$, and $\Aa$ a non-empty proper subset of $\F_q$. Let $\M$ be a random $m \times n$ matrix of rank $r$ over $\F_q$ taken with unfiorm distribution. It was proved recently by Sanna that as $m,n \to \infty$ and $r,q,\Aa$ are fixed, the number of entries of $\M$ in $\Aa$ approaches a normal distribution. The question was raised as to whether or not one can still obtain a central limit theorem of some sort when $r$ goes to infinity in a way controlled by $m$ and $n$. In this paper we answer this question affirmatively.
\end{abstract}

\section{Introduction}\label{int} 

Denote by $\F_q$ the finite field of order $q$. For a matrix $\M$ over $\F_q$, denote by $\wt(\M)$ the \emph{weight} of $\M$ over $\F_q$, that is, the number of nonzero entries of $\M$.

For positive integers $m,n,r$, denote by $\Fmn$ the set of $m \times n$ matrices of rank $r$ over $\F_q$. After providing a formula for the mean value of $\wt(\M)$ as $\M$ is taken at random uniformly from the set $\Fmn$, Migler, Morrison and Ogle \cite{Mig} suggested that as $m,n \to \infty$ and $r,q$ are fixed, an appropriate scaling of $\wt(\M)$ approaches a normal distribution. This claim was proved recently by Sanna \cite{Sanna}, building upon his previous work \cite{Sanna1} which proved the claim partially under the condition that $q=2$ and $m/n$ converges to a positive real number.  

Sanna's proof is based on Fourier analysis over $\F_q$ and the M\"obius inversion formula, which is quite complex. While the method works nicely for $m,n \to \infty$ and $r,q$ are fixed,  when $r \to \infty$, however, the method runs into serious difficulties. Hence Sanna raised the question: \emph{Can one still obtain a central limit theorem of some sort when $r$ goes to infinity in a way controlled by $m$ and $n$?} (see \cite[Remark 5.1]{Sanna}). 

The purpose of this paper is to answer this question in the affirmative. Similar to Sanna's result \cite[Theorem 1]{Sanna}, our main result can also be stated for more general weight functions. To describe the main result, we need some notations. 

For every $\Aa \subset \F_q$ and for any matrix $\M$ over $\F_q$, denote by $\ct(\M)$ the number of entries of $\M$ that belong to $\Aa$. Also define 
$$\gaa:=q^{-1}\#\Aa-\bo_{\Aa}(0),$$
here $\#A$ denotes the cardinality of any finite set $A$, and,  
\[\bo_{\Aa}(x)=\left\{\begin{array}{lll}1&:& x \in \Aa;\\
0&:& x \notin \Aa .\end{array}\right.\]
It is easy to see that $q^{-1} \le  |\gamma_{\Aa}(q)|\le 1-q^{-1}$ if $\Aa$ is a non-empty proper subset of $\F_q$. 

For any positive integers $m,n,r$, define 
$$\mu_{\Aa}(q,m,n,r):=mn\left(q^{-1}\#\Aa-q^{-r}\gaa \right),$$
\begin{eqnarray*}
\sigma_{\Aa}^2(q,m,n,r):=&mn\left(q^{-1}\#\Aa-q^{-r}\gaa\right)\left(1-q^{-1}\#\Aa+q^{-r}\gaa\right)\\
&+mn(m+n-2)q^{-r}(1-q^{-r})\gaa^2. 
\end{eqnarray*}

Now we state the main result of this paper. 

\begin{theorem}\label{MainThm}
    Let $\emptyset \subsetneq \Aa \subsetneq \F_q$ be fixed and let $\M$ be taken at random with uniform distribution from the set $\Fmn$. Assume that as $m,n \to \infty$ such that $\min\{m,n\}-r \to \infty$ and one of the following three conditions holds: 

(i) $\lim_{m,n \to \infty} \frac{q^r}{\min \{m,n\}}=0$,  
    
(ii) $\lim_{m,n \to \infty}\frac{q^r}{(m+n)^a}=\infty$ for any fixed $a>0$, 

(iii) $m \asymp n$ and $\lim_{m\to \infty} \frac{q^r}{m}=\infty$,  

\noindent then the term 
    $$\frac{\ct(\M)-\mu_{\Aa}(q,m,n,r)}{\sqrt{\sigma_{\Aa}^2(q,m,n,r)}}$$
    converges in distribution to a standard normal random variable. 
    \end{theorem}

\begin{remark} 1). The term $\mu_{\Aa}(q,m,n,r)$ is the same as that appeared in \cite{Sanna}. 

2). Under Condition (i), we have 
$$\sigma_{\Aa}^2(q,m,n,r) \sim \gaa^2 q^{-r}(1-q^{-r}) (m+n)mn, \quad \mbox{ as } m,n \to \infty.$$
This estimate of $\sigma_{\Aa}^2(q,m,n,r)$ is essentially the formula appearing in \cite{Sanna}, which dealt with the special case of (i) that $m,n \to \infty$ and $r,q$ are fixed. So in this case Theorem 1 extends \cite[Theorem 1.1]{Sanna} in the sense that we can allow $r \to \infty$ slowly with respect to $m, n$ as $m,n \to \infty$ (see Condition (i)). 

3). Under Condition either (ii) or (iii), then $r \to \infty$ as $m,n \to \infty$, and we have 
$$\sigma_{\Aa}^2(q,m,n,r) \sim q^{-1}\#\Aa \, (1-q^{-1}\#\Aa)mn, \quad \mbox{ as } m,n \to \infty.$$
This estimate of $\sigma_{\Aa}^2(q,m,n,r)$ is quite different from that in \cite{Sanna}, showing that under (ii) or (iii), the term $\ct(\M)$ has a quite different behavior (with a different variance), though after nomalization, it still converges to a standard normal random distribution.  
\end{remark}

To prove Theorem \ref{MainThm}, we use the moment method. Here our method differs significantly from that of Sanna: we compute all the moments directly via a graph method, which helps us identify the main terms and error terms according to different patterns of graphs. The graph method we use in this paper is reminiscent to those in \cite{MR4009185, MR4519869,MR3273060}, though the techniques involved here are different. One complex feature of this paper is that how each graph is decomposed into connected components plays an important role in the proof of the final result. 

This paper is organized as follows. In Section \ref{fol} we adopt Sanna's strategy and convert the original problem into that of the product of $m \times r$ and $r \times n$ matrices. So to prove Theorem \ref{MainThm}, it suffices to compute all the moments (see Theorem \ref{M} in Section \ref{fol}). It turns out that the method works even when $r \to \infty$. In Section \ref{exp} we compute the first two moments ($\l=1,2$) directly. In Section \ref{moment}, we set up the problem for the graph method and prove some crucial lemmas. Then finally in Section \ref{high}, we consider graph decomposition into connected components and identify the main terms and the error terms, hence proving Theorem \ref{M}. This concludes the paper.

\section{From $\Fmn$ to $\F_q^{m \times r} \times \F_q^{r \times n}$}\label{fol}

Denote by $\F_q^{m \times n}$ the set of $m \times n$ matrices over $\F_q$. To prove Theorem \ref{MainThm}, we first need to extend \cite[Lemma 4.2]{Sanna} to the case that $r \to \infty$. This turns out to be quite straightforward. 

\begin{lemma}\label{P}
    Let $\mathbf{M} \in \F_q^{m \times n,r}, \mathbf{X} \in \F_q^{m \times r}, \mathbf{Y} \in \F_q^{r \times n}$ be independent random matrices uniformly distributed in their respective spaces. Then
$$\sum_{\NN \in \F_q^{m \times n}} \big|\P[\mathbf{XY}=\NN]-\P[\mathbf{M}=\NN] \big| \to 0$$
    as $m,n \to +\infty$ such that $\min\{m,n\}-r \to \infty$.
\end{lemma}
\begin{proof}
 Following the proof of \cite[Lemma 4.2]{Sanna} closely, in order to prove Lemma \ref{P}, we just need to show that
    \begin{eqnarray} \label{2:ineq1} A=1-\frac{\prod_{i=0}^{r-1}(q^m-q^i)(q^n-q^i)}{q^{mr}\cdot q^{rn}} \to 0\end{eqnarray}
    as $\min\{m,n\}-r \to \infty$.

It is easy to see that
\[A=1-\prod_{i=0}^{r-1}\left(1-q^{i-m}\right)\left(1-q^{i-n}\right)>0.\]
On the other hand, applying the inequality below, which can be proved easily by induction,
\[\prod_{i=1}^m (1-x_i) \ge 1-x_1-\cdots-x_m, \qquad \forall x_i \in (0,1), \]
we obtain
\begin{eqnarray*}
A &\le& \sum_{i=0}^{r-1} q^{i-m}+\sum_{i=0}^{r-1} q^{i-n} <\left(q^{r-1-m}+q^{r-1-n}\right) \sum_{i=0}^{\infty} q^{-i}\\
&=&\left(q^{r-m}+q^{r-n}\right) \frac{1}{1-q^{-1}} \to 0,
\end{eqnarray*}
as $\min\{m,n\} -r \to \infty$. So (\ref{2:ineq1}) is proved. This completes the proof of Lemma \ref{P}.

\end{proof}

Fix a nonempty $\Aa \subsetneq \F_q$ and for the sake of brevity, let
\[\cts(\N):=\frac{\ct(\N)-\mua}{\sqrt{\siga}} \]
for any $\N \in \F_q^{m \times n}$.

Let $\M \in \Fmn, \X \in \F_{q}^{m \times r}, \Y \in \F_{q}^{r \times n}$ be independent random matrices uniformly distributed in their respective spaces. Thanks to Lemma \ref{P} and following the idea of \cite{Sanna}, for every real number $t$, we have that
\begin{eqnarray} \label{2:conv} 
\big |\P[\cts(\M) \le t] - \P[\cts(\X \Y) \le t]\big |&=&\bigg | \sum_{\substack{\N \in \F_q^{m \times n} \nonumber\\
\cts(\N) \le t}} \left(\P[\M=\N]-\P[\X \Y=\N]\right) \bigg | \\
&\le & \sum_{\N \in \F_q^{m \times n}} \big |\P[\M=\N]-\P[\X \Y=\N]\big | \to 0,
\end{eqnarray}
as $\min\{m,n\}-r \to \infty$. So to prove Theorem \ref{MainThm}, it suffices to study $\cts(\X \Y)$ as $\X \in \F_{q}^{m \times r}, \Y \in \F_{q}^{r \times n}$ are independent random matrices uniformly distributed in their respective spaces. We will use the moment method and prove the following: 

\begin{theorem}\label{M}
    Let $\mathbf{X} \in \F_q^{m \times r}$ and $\mathbf{Y} \in \F_q^{r \times n}$ be uniformly and independently distributed in their respective spaces. As $m,n \to \infty$, assume one of the following three conditions holds: 

(i) $\lim_{m,n \to \infty} \frac{q^r}{\min \{m,n\}}=0$, or 
    
(ii) $\lim_{m,n \to \infty}\frac{q^r}{(m+n)^a}=\infty$ for any fixed $a>0$, or 

(iii) $m \asymp n$ and $\lim_{m\to \infty} \frac{q^r}{m}=\infty$. 

\noindent Then for any positive integer $\l$, 
    $$\E\left[\cts(\mathbf{XY})^\l\right]=\begin{cases}
        0 &(\l=1)\\
        1 &(\l=2)\\
        o_\l(1) &(\l \geq 3 \text{ and odd})\\
        (\l-1)!!+o_\l(1) &(\l \geq 4 \text{ and even}).
    \end{cases}$$
    Here the subscript $\l$ in the little-o notation means that the implied constant depends only on $\l$.
\end{theorem}

If Theorem \ref{M} is proved, since the sequence of the Gaussian moments $m_\l=(\l-1)!!$ satisfies the Carleman's condition (see \cite{MR3930614})
$$\sum_{\l=1}^\infty m_{2\l}^{-1/2\l}=+\infty,$$
we can conclude that the quantity $\cts(\mathbf{XY})$ converges in distribution to a standard normal random variable by the moment convergence theorem. Then by (\ref{2:conv}), as $\min\{m,n\}-r \to \infty$, the term $\cts(\mathbf{M})$ with $\mathbf{M} \in \F_q^{m \times n,r}$ chosen uniformly would follow the same distribution asymptotically. So this proves Theorem \ref{MainThm}. 

Thereby it remains to prove Theorem \ref{M}.

\section{Expectation and variance of $\ct(\X\Y)$} \label{exp}

For any integers $a,b$ with $a<b$, denote $[a \isep b]:=[a,b] \cap \Z$. 

Let $\mathbf{X} \in \F_q^{m \times r}$, $\mathbf{Y} \in \F_q^{r \times n}$ be independent random matrices uniformly distributed in their respective spaces. 
For $i \in [1\isep m]$ and $j \in [1\isep n]$, denote by $\x_i$ the $i$-th row of $\mathbf{X}$, and $\y_j^T$ the $j$-th column of $\mathbf{Y}$. Here $\x_i,\y_j \in \F_{q}^r$ are uniformly and independently distributed in the space. 

It is easy to see that the quantity $\ct(\mathbf{XY})$ can be expanded as
$$\ct(\mathbf{XY})=\sum_{i=1}^m\sum_{j=1}^n\mathbbm{1}_\A(\x_i\cdot \y_j).$$
We first prove Theorem \ref{M} for the cases $\l \in \{1,2\}$, that is, 
\begin{lemma}\label{musigma}
Under the above setting, we have
    \begin{enumerate}
        \item[1).] $\E[\ct(\mathbf{XY})]=\mu_\A(q,m,n,r)$;
        \item[2).] $\Var[\ct(\mathbf{XY})]=\sigma_\A^2(q,m,n,r)$.
    \end{enumerate}
\end{lemma}
\begin{proof}
1). Let $\x$ be a fixed (deterministic) vector in $\F_q^r$. For any $j \in [1\isep n]$, if $\x=\0$, then $\x\cdot\y_j$ is always 0, while if $\x \neq \0$, then $\x\cdot\y_j$ runs uniformly over $\F_q$ as $\y_j$ varies over $\F_q^r$. Hence,
    $$\P[\x\cdot\y_j \in \A]=\begin{cases}
        \mathbbm{1}_\A(0) &(\x=\0)\\
        q^{-1}\#\A &(\x \neq \0).
    \end{cases}$$
    This implies, for any $i \in [1\isep m], j \in [1\isep n]$,
    \begin{align}\label{P0}
        \P\left[\x_i\cdot\y_j \in \A \right]&=\sum_{\x \in \F_q^r}\P[\x_i=\x]\P[\x_i\cdot\y_j\in \A | \x_i=\x]\nonumber\\
     &=q^{-r}\mathbbm{1}_\A(0)+(1-q^{-r})q^{-1}\#\A\nonumber\\
    &=q^{-1}\#\A-q^{-r}\gaa.
    \end{align}
So we have
    \begin{align*}
    \E[\ct(\mathbf{XY})]&=\sum_{i=1}^m\sum_{j=1}^n\E[\mathbbm{1}_\A(\x_i\cdot\y_j)]\\
    &=\sum_{i=1}^m\sum_{j=1}^n\P[\x_i\cdot\y_j \in \A]\\
    &=mn(q^{-1}\#\A-q^{-r}\gaa) =\mu_\A(q,m,n,r).
    \end{align*}

2). We have
    \begin{align}\label{V}
        \Var\left[\ct(\mathbf{XY})\right]&=\E\left[\ct(\mathbf{XY})^2\right]-\left(\E[\ct(\mathbf{XY})]\right)^2\nonumber\\
        &=\E\left[\left(\sum_{i=1}^m\sum_{j=1}^n \mathbbm{1}_\A(\x_i\cdot\y_j)\right)^2\right]-\left(\sum_{i=1}^m\sum_{j=1}^n\E[\mathbbm{1}_\A(\x_i\cdot\y_j)]\right)^2\nonumber\\
        &=\sum_{i,i'=1}^m\sum_{j,j'=1}^n W_{iji'j'}
    \end{align}
    where
    $$W_{iji'j'}:=\E[\mathbbm{1}_\A(\x_i\cdot\y_j)\mathbbm{1}_\A(\x_{i'}\cdot\y_{j'})]-\E[\mathbbm{1}_\A(\x_i\cdot\y_j)]\E[\mathbbm{1}_\A(\x_{i'}\cdot\y_{j'})].$$
    We now evaluate the quantity $W_{iji'j'}$ by dividing into the following four cases:

    \textbf{Case 0:} $i \neq i', j \neq j'$

    In this case, since all the row vectors $\x_i,\y_j,\x_{i'},\y_{j'}$ are independently distributed in $\F_q^r$, we have 
    $$W_{iji'j'}=0.$$

    \textbf{Case 1:} $i=i'$ and $j \neq j'$

    In this case the vectors $\x_i,\y_j$ and $\y_{j'}$ are independently distributed in $\F_q^r$. We have, by (\ref{P0}),
    \begin{align*}
        \E\left[\mathbbm{1}_\A(\x_i\cdot\y_j)\mathbbm{1}_\A(\x_{i}\cdot\y_{j'})\right]&=\E\left[\E[\mathbbm{1}_\A(\x_i\cdot\y_j)\mathbbm{1}_\A(\x_i\cdot\y_{j'})|\x_i] \right]\\
&=\E\left[\E[\mathbbm{1}_\A(\x_i\cdot\y_j)|\x_i] \E[\mathbbm{1}_\A(\x_i\cdot\y_{j'})|\x_i]\right]\\
        &=q^{-r}\sum_{\x \in \F_q^r}\P[\x\cdot\y_j \in \A \wedge \x\cdot\y_{j'}\in \A]\\
        &=q^{-r}\sum_{\x \in \F_q^r}\P[\x\cdot\y_j \in \A]\P[\x\cdot\y_{j'}\in \A]\\
        &=q^{-r}\mathbbm{1}_\A(0)+(1-q^{-r})(q^{-1}\#\A)^2.
    \end{align*}
    Hence
    \begin{align*}
    W_{iji'j'}&=q^{-r}\mathbbm{1}_\A(0)+(1-q^{-r})(q^{-1}\#\A)^2-[q^{-r}\mathbbm{1}_\A(0)+(1-q^{-r})q^{-1}\# \A]^2\\
    &=q^{-r}(1-q^{-r})[\mathbbm{1}_\A(0)-2q^{-1}\# \A\mathbbm{1}_\A(0)+(q^{-1}\#\A)^2]\\
    &=q^{-r}(1-q^{-r})\gaa^2.
    \end{align*}
    \textbf{Case 2:} $i \neq i'$ and $j=j'$

    This case is similar to \textbf{Case 1} with the roles of $i$ and $j$ swapped. Since dot product is commutative, we easily see that in this case we also have
    $$W_{iji'j'}=q^{-r}(1-q^{-r})\gaa^2.$$
    \textbf{Case 3}: $i=i'$ and $j=j'$

    In this case we have
    $$\E[\mathbbm{1}_\A(\x_i\cdot\y_j)\mathbbm{1}_\A(\x_{i'}\cdot\y_{j'})]=\E[\mathbbm{1}_\A(\x_i\cdot\y_j)^2]=\E[\mathbbm{1}_\A(\x_i\cdot\y_j)]=q^{-1}\#\A-q^{-r}\gaa$$
    by (\ref{P0}).

    Therefore
    \begin{align*}
    W_{iji'j'}&=q^{-1}\#\A-q^{-r}\gaa-(q^{-1}\# \A-q^{-r}\gaa)^2\\
    &=(q^{-1}\#\A-q^{-r}\gaa)(1-q^{-1}\#\A+q^{-r}\gaa).
    \end{align*}

A simple counting shows that there are $mn(m-1)(n-1), mn(n-1), mn(m-1)$ and $mn$ choices of $(i,j,i',j')$ in \textbf{Cases 0,1,2} and \textbf{3} respectively. Combining all these and putting into (\ref{V}) then yields
\begin{align*}
    &\Var[\ct(\mathbf{XY})]\\
    &=mn(m+n-2)q^{-r}(1-q^{-r})\gaa^2+mn\left(q^{-1}\#\A-q^{-r}\gaa\right)\left(1-q^{-1}\#\A+q^{-r}\gaa\right)\\
    &=\sigma_\A^2(q,m,n,r)
\end{align*}
as desired.
\end{proof}

\section{Estimation of Higher Order Moments}\label{moment} 

Now we prove Theorem \ref{M} for $\l \ge 3$. 
\subsection{Problem Set-up}

Given positive integers $a$ and $b$, denote by $\Gamma_{a,b}$ the set of all maps $\gamma: [1\isep a] \to [1\isep b]$. 

By Lemma \ref{musigma}, we may write
\begin{align}
    \E[\cts(\mathbf{XY})^\l]&=\E\left[\left(\frac{\sum_{i=1}^m\sum_{j=1}^n \Big\{\mathbbm{1}_\A(\x_i\cdot\y_j)-\E[\mathbbm{1}_\A(\x_i\cdot\y_j)]\Big\}}{\sigma_\A(q,m,n,r)}\right)^\l\right]\nonumber\\
    &=\sum_{\gamma \in \Gamma_{\l,m}}\sum_{\tau \in \Gamma_{\l,n}}
    \frac{\E\left[\prod_{k=1}^\l \Big\{\mathbbm{1}_\A(\x_{\gamma(k)}\cdot\y_{\tau(k)})-\E[\mathbbm{1}_\A (\x_{\gamma(k)}\cdot\y_{\tau(k)})]\Big\}\right]}{(\sigma_\A^2(q,m,n,r))^{\l/2}}
    \nonumber\\
    &=:\frac{1}{(\sigma_\A^2(q,m,n,r))^{\l/2}}\sum_{\gamma \in \Gamma_{\l,m}}\sum_{\tau \in \Gamma_{\l,n}}W_{\gamma\tau}, \nonumber
    \end{align}
    where for any $\gamma \in \Gamma_{\l,m}, \tau \in \Gamma_{\l,n}$, 
    \begin{align} \label{E3}
    W_{\gamma\tau}:&=\E\left[\prod_{k=1}^\l \Big\{\mathbbm{1}_\A(\x_{\gamma(k)}\cdot\y_{\tau(k)})-\E[\mathbbm{1}_\A(\x_{\gamma(k)}\cdot\y_{\tau(k)})]\Big\}\right].
    \end{align}
    For any positive integer $a$, denote by $\Sigma_a$ the set of permutations on $[1\isep a]$. It is then easy to see that for any $\rho \in \Sigma_m$ and any $\pi \in \Sigma_n$, 
    $$W_{(\rho \circ \gamma)(\pi \circ \tau)}=W_{\gamma\tau}.$$
    Moreover, define $\gamma' \sim \gamma$ whenever $\gamma'=\rho \circ \gamma$ for some $\rho \in \Sigma_m$ and $\tau' \sim \tau$ whenever $\tau'=\pi \circ \tau$ for some $\pi \in \Sigma_n$. This defines equivalence relations on $\Gamma_{\l,m}$ and $\Gamma_{\l,n}$ respectively. Now for any $\gamma \in \Gamma_{\l,m}$ and $\tau \in \Gamma_{\l,n}$, define
    $$U_\gamma:=\left\{\gamma(k): k \in [1\isep \l] \right\}, \quad V_\tau=\left\{ \tau(k): k \in [1\isep \l] \right\},$$
    $$u_\gamma=\#U_\gamma, \quad v_\tau=\#V_\tau.$$
   It is easy to see that 
    $$\#[\gamma]=\frac{m!}{(m-u_\gamma)!}, \quad \#[\tau]=\frac{n!}{(n-v_\tau)!}.$$
    Denote $\Gamma_\l:=\Gamma_{\l,m}/\Sigma_m \times \Gamma_{\l,n}/\Sigma_n$. Then we have    
    \begin{equation}\label{E4}
    \E[\cts(\mathbf{XY})^\l]=\frac{1}{(\sigma_\A^2(q,m,n,r))^{\l/2}}\sum_{(\gamma,\tau) \in \Gamma_\l}\frac{m!}{(m-u_\gamma)!}\frac{n!}{(n-v_\tau)!}W_{\gamma\tau},
    \end{equation}
    where $W_{\gamma\tau}$ is defined in (\ref{E3}). 
    
\subsection{Analysis of $W_{\gamma\tau}$}
For each $(\gamma,\tau) \in \Gamma_\l$, we define an undirected bipartitie graph $G_{\gamma \tau}=(U_{\gamma},V_{\tau},E_{\gamma \tau})$ as follows: the vertex set is $U_{\gamma} \cup V_{\tau}$ and the edge set is the multi-set
\[E_{\gamma \tau}:=\left\{\left\{\overline{\gamma(k) \tau(k)}: k \in [1 \isep \l]\right\}\right\}. \]
We also define an undirected bipartitie graph $G_{\gamma \tau}'=(U_{\gamma},V_{\tau},E_{\gamma \tau})$ where the vertex set is $U_{\gamma} \cup V_{\tau}$ but the edge set is the set 
\[E_{\gamma \tau}':=\left\{\overline{\gamma(k) \tau(k)}: k \in [1 \isep \l]\right\}. \]
So $G_{\gamma \tau}$ is a multi-graph, with possibly multiple edges from a vertice in $U_{\gamma}$ to a vertice in $V_{\tau}$ and $\#E_{\gamma \tau}=\l$; $G_{\gamma \tau}'$ is a simple graph, with at most one edge from a vertice in $U_{\gamma}$ to a vertice in $V_{\tau}$. If we denote $\ep_{\gamma \tau}':=\#E_{\gamma \tau}'$, then clearly 
\[ \ep_{\gamma\tau}' \leq \#E_{\gamma \tau}=\l, \quad \forall \gamma, \tau. \]

We first have the following preliminary estimation on $W_{\gamma\tau}$.
\begin{lemma}\label{W0}
    Let $(\gamma,\tau) \in \Gamma_\l$. Then
    $$W_{\gamma\tau}=\begin{cases}
        0 &(\l=1),\\
        O_\l(q^{-r}) &(\l \geq 2 \text{ and the graph } G_{\gamma \tau} \text{ has at least one simple edge}),\\
        O_\l(1) &(\text{otherwise}).
    \end{cases}$$
\end{lemma}

Before proving Lemma \ref{W0}, we need the following result.
\begin{lemma}\label{W'}
    For any $(\gamma,\tau) \in \Gamma_\l$, we have 
    $$\E\left[\prod_{k=1}^\l \mathbbm{1}_\A(\x_{\gamma(k)}\cdot \y_{\tau(k)})\right]=\left(q^{-1}|\A|\right)^{\ep_{\gamma\tau}'}+O_\l(q^{-r}).$$
\end{lemma}
\begin{proof}
    Denote
    $$W_{\gamma\tau}:=\E\left[\prod_{k=1}^\l \mathbbm{1}_\A(\x_{\gamma(k)}\cdot \y_{\tau(k)})\right].$$
    If there are multi-edges in $G_{\gamma \tau}$, say $\overline{\gamma(k_2)\tau(k_2)}=\overline{\gamma(k_1)\tau(k_1)}$ for some $k_2 > k_1$, that is $\gamma(k_2)=\gamma(k_1)$ and $\tau(k_2)=\tau(k_1)$, then clearly
    \[\mathbbm{1}_\A(\x_{\gamma(k_1)}\cdot \y_{\tau(k_1)}) \mathbbm{1}_\A(\x_{\gamma(k_2)}\cdot \y_{\tau(k_2)})=\mathbbm{1}_\A(\x_{\gamma(k_1)}\cdot \y_{\tau(k_1)}), \]
    we can remove the index $k_2$ in the computation of $W_{\gamma \tau}$, or equivalently we can remove the edge $\overline{\gamma(k_2)\tau(k_2)}$ from $G_{\gamma \tau}$ without affecting the computation of $W_{\gamma \tau}$. So let us assume that there is no multi-edge in $G_{\gamma \tau}$, hence $\ep_{\gamma\tau}'=\l$ and for any distinct $k_1,k_2 \in [1\isep \l]$, either $\gamma(k_1) \neq \gamma(k_2)$ or $\tau(k_1) \neq \tau(k_2)$.

    We may first write
    \begin{equation}\label{W2}
        W_{\gamma\tau}=\E\left[\E\left[\prod_{k=1}^\l \mathbbm{1}_\A(\x_{\gamma(k)}\cdot \y_{\tau(k)})\bigg|\x_{\gamma(1)}, \x_{\gamma(2)}\cdots, \x_{\gamma(\l)} \right] \right].
    \end{equation}

    For each $t \in V_\tau$, define $S_t:=\tau^{-1}(t)$ and $s_t:=\#S_t$. By assumption, $\gamma$ is one-to-one on $S_t$ for each $t$. Moreover, $\cup_{t \in V_\tau} S_t=[1\isep \l]$ and so $\sum_{t \in V_\tau} s_t=\l$.

    Since $\tau(k)=t$ for any $k \in S_t$, we can write (\ref{W2}) as
    \begin{equation}\label{W3}
        W_{\gamma\tau}=\E\left[\prod_{t \in V_\tau}\P\left[\x_{\gamma(k)}\cdot\y_t \in \A \quad \forall k \in S_t\bigg|\x_{\gamma(1)}, \x_{\gamma(2)}\cdots, \x_{\gamma(\l)}\right]\right].
    \end{equation}

    To compute the inner term $\P\left[\x_{\gamma(k)}\cdot\y_t \in \A \quad \forall k \in S_t\bigg|\x_{\gamma(1)}, \x_{\gamma(2)}\cdots, \x_{\gamma(\l)}\right]$, suppose $\x_{\gamma(k)}=\a_k$ for each $k \in S_t$ and the vectors $\a_k \in \F_q^r$  for $k \in S_t$ are linearly independent over $\F_q$. Then for any fixed $b_k \in \A$ where $k \in S_t$, the system of equations
    $$\a_k \cdot \y=b_k, \quad \forall k \in S_t,$$
    can be rewritten as
    $$\mathbf{A} \y^T=\b,$$
    where 
    $$
    \mathbf{A}=\begin{bmatrix} 
    \a_1\\
    \vdots\\
    \a_{s_t}
    \end{bmatrix}, \quad \b=\begin{bmatrix} b_1\\
    \vdots\\
    b_{s_t}
    \end{bmatrix}.$$
More generally, the system
   $$\a_k\cdot \y \in \A, \quad \forall k \in S_t$$
   has $q^{r-s_t}|\A|^{s_t}=q^r(q^{-1}|\A|)^{s_t}$ solutions for $\y \in \F_q^r$. This implies that
   \begin{equation}\label{P1}
   \P[\x_{\gamma(k)}\cdot\y_t \in \A \quad \forall k \in S_t|\x_{\gamma(k)}=\a_k \quad \forall k \in S_t]=\frac{q^r(q^{-1}|\A|)^{s_t}}{q^r}=(q^{-1}|\A|)^{s_t}, 
   \end{equation}
   whenever the vectors $\a_k$ for $k \in S_t$ are linearly independent over $\F_q$. 

   On the other hand, let 
   \begin{align*}
   p&:=\P\big[\{\x_s: s \in U_\gamma\} \text{ linearly independent}\big]. 
   \end{align*}
   Since $\x_1,\cdots,\x_{\l}$ are all uniformly and independently distribution in $\F_q^r$, we can obtain 
   \begin{align*}
   p   &=\P\big[[\x_s]_{s \in U_\gamma} \in \F_q^{r \times u_\gamma,u_\gamma}\big] 
   =\frac{\prod_{i=0}^{u_\gamma-1}(q^r-q^i)}{q^{ru_\gamma}}\\  
   &=\prod_{i=0}^{u_\gamma-1}(1-q^{i-r}) 
   =1+O_\l(q^{-r}).
   \end{align*}
   Hence, using (\ref{P1}) and (\ref{W3}) we have 
    \begin{align*}
        W_{\gamma\tau}&=p \cdot \left(\prod_{t \in V_\tau}\P\big[\x_{\gamma(k)}\cdot\y_t \in \A \quad \forall k \in S_t|\{\x_s: s \in U_\gamma\} \text{ linearly independent}\big]\right)+\\
        &(1-p)\cdot \left(\prod_{t \in V_\tau}\P[\x_{\gamma(k)}\cdot\y_t \in \A \quad \forall k \in S_t|\{\x_s: s \in U_\gamma\} \text{ linearly dependent}]\right)\\
        &=(1+O_\l(q^{-r}))\prod_{t \in V_\tau} (q^{-1}|\A|)^{s_t}+O_\l(q^{-r})\\
        &=(q^{-1}|\A|)^\l+O_\l(q^{-r}).
    \end{align*}
As $\ep_{\gamma \tau}'=\l$, this proves Lemma \ref{W'} when $G_{\gamma \tau}$ is a simple graph. By the reduction process as described in the beginning, this completes the proof of Lemma \ref{W'}. 
\end{proof}

\begin{proof}[Proof of Lemma \ref{W0}]
    When $\l=1$, it is obvious that 
    $$W_{\gamma\tau}=\E[\mathbbm{1}_\A(\x_{\gamma(1)}\cdot\y_{\tau(1)})-\E[\mathbbm{1}_\A(\x_{\gamma(1)}\cdot\y_{\tau(1)})]]=0.$$ 

    When $\l \geq 2$ and all edges in $G_{\gamma\tau}$ are multiple, it is also trivial that $|W_{\gamma \tau}| \ll_\l 1$, as each inner term in the expression (\ref{E3}) of $W_{\gamma \tau}$ is bounded by 2. 

    Let us assume that $\l \geq 2$ and at least one edge in $G_{\gamma\tau}$ is simple. Without loss of generality, assume $e_1=\overline{\gamma(1)\tau(1)}$ is a simple edge in $G_{\gamma\tau}$.

    For any subset $\S \subset [1\isep \l]$, let $(\gamma_\S,\tau_\S)$ denote the restriction of $(\gamma,\tau)$ on $\S$, and $G_{\gamma_\S\tau_\S}$ the bipartite multi-graph resulting from $(\gamma_{\S},\tau_{\S})$. In particular, if $1 \in \S$, then $e_1$ is also a simple edge in $G_{\gamma_\S\tau_\S}$.

    Now let $\S \subset [2\isep \l]$ and denote $\widetilde{\S}:=\S \cup \{1\}$. Then $\ep_{\gamma_{\widetilde{\S}}\tau_{\widetilde{\S}}}'=\ep_{\gamma_{\S}\tau_{\S}}'+1$ since the edge $e_1$ does not appear in $G_{\gamma_{\S}\tau_{\S}}$. We then have
    \begin{align*}
    &\E\left[\Big\{\mathbbm{1}_\A(\x_{\gamma(1)}\cdot\y_{\gamma(1)})-\E[\mathbbm{1}_\A(\x_{\gamma(1)}\cdot\y_{\gamma(1)})] \Big\}\prod_{k \in \S}\mathbbm{1}_\A(\x_{\gamma(k)}\cdot\y_{\tau(k)})\right]\\
    &=W_{\gamma_{\widetilde{\S}}\tau_{\widetilde{\S}}}-W_{\gamma_{\S}\tau_{\S}}\E[\mathbbm{1}_\A(\x_{\gamma(1)}\cdot\y_{\tau(1)})]\\
    &=\left(q^{-1}|\A|\right)^{\ep_{\gamma_{\widetilde{\S}}\tau_{\widetilde{\S}}}'}+O_\l(q^{-r})-\big[(q^{-1}|\A|)^{\ep_{\gamma_{\S}\tau_{\S}}'}+O_\l(q^{-r})\big]\left(q^{-1}|\A|+q^{-r}\alpha_\A\right)\\
    &=\left(q^{-1}|\A|\right)^{\ep_{\gamma_{\widetilde{\S}}\tau_{\widetilde{\S}}}'}-\left(q^{-1}|\A|\right)^{\ep_{\gamma_{\S'}\tau_{\S'}}'+1}+O_\l(q^{-r})\\
    &=O_\l(q^{-r}).
    \end{align*}
    By (\ref{E3}) and the inclusion-exclusion formula, we obtain
    \begin{align*}
        W_{\gamma\tau}&= \E\left[\Big\{\mathbbm{1}_\A(\x_{\gamma(1)}\cdot\y_{\tau(1)})-\E[\mathbbm{1}_\A(\x_{\gamma(1)}\cdot\y_{\tau(1)})]\Big\}\prod_{k=2}^\l \Big\{\mathbbm{1}_\A(\x_{\gamma(k)}\cdot\y_{\tau(k)})-\E[\mathbbm{1}_\A(\x_{\gamma(k)}\cdot\y_{\tau(k)})]\Big\}\right]\\
        &=\sum_{\S \subset [2 \isep \l]} \E\left[\Big\{\mathbbm{1}_\A(\x_{\gamma(1)}\cdot\y_{\gamma(1)})-\E[\mathbbm{1}_\A(\x_{\gamma(1)}\cdot\y_{\gamma(1)})]\Big\}\prod_{k \in \S}\mathbbm{1}_\A(\x_{\gamma(k)}\cdot\y_{\tau(k)})\right]\times\\
        &(-1)^{\l-1-\#\S}\prod_{k' \in [2\isep \l] \setminus \S}\E\left[\mathbbm{1}_\A(\x_{\gamma(k')}\cdot\y_{\tau(k')})\right]\\
        &=\sum_{\S \subset [2\isep \l]}O_\l(q^{-r}) =O_\l(q^{-r})
        \end{align*}
    as desired. This completes the proof of Lemma \ref{W0}.
\end{proof}

Decompose the multi-graph $G_{\gamma\tau}$ into connected components 
$$G_{\gamma \tau}=\sqcup_{i=1}^\kappa G_{\gamma_i\tau_i}.$$ 
Here $\kappa:=\kappa_{\gamma \tau}$ is the number of connected components, and for each $i$, $G_{\gamma_i\tau_i}=(U_{\gamma_i},V_{\tau_i},E_{\gamma_i \tau_i})$ is the $i$-th component which is also a bipartie multi-graph arising from $(\gamma_i,\tau_i) \in \Gamma_{\l_i}$. We have the relations  
\[U_{\gamma}=\sqcup_i U_{\gamma_i}, \quad V_{\tau}=\sqcup_i V_{\tau_i}, \quad E_{\gamma \tau}=\sqcup_i E_{\gamma_i \tau_i}, \quad E_{\gamma \tau}'=\sqcup_i E_{\gamma_i \tau_i}', \]
\[\ep_{\gamma_i \tau_i}'=\#E_{\gamma_i \tau_i}' \le \# E_{\gamma_i \tau_i}=\l_i \quad \forall i, \]
and 
 $$\sum_{i=1}^\kappa \l_i=\l.$$. 
Due to the fact that $\x_i$ and $\x_{i'}$ for $i \neq i'$ (resp. $\y_j$ and $\y_{j'}$ for $j \neq j'$) are uniformly and independently distributed in $\F_q^r$, we have 
\begin{equation}\label{Prod}
W_{\gamma\tau}=\prod_{i=1}^\kappa W_{\gamma_i\tau_i},
\end{equation}
where each $W_{\gamma_i \tau_i}$ can be estimated by Lemma \ref{W0}. 

Each connected component $G_{\gamma_i \tau_i}$ falls into one of the following three types:
\begin{enumerate}
    \item[$T_0$:] The component consists of only one edge, which is simple;
    \item[$T_1$:] The component has at least two edges, and at least one edge is simple;
    \item[$T_2$:] All edges in the component are multiple edges.
\end{enumerate}
For each $(\gamma, \tau) \in \Gamma_{\l}$, denote by $\kappa_{i}:=\kappa_{\gamma \tau,i}$ the number of connected components of Type $T_i$ in $G_{\gamma\tau}$ for $i=0,1,2$, so we have 
$$\kappa=\sum_{i=0}^2\kappa_{i}.$$ 
According to (\ref{Prod}) and applying Lemma \ref{W0} to each connected component $G_{\gamma_i \tau_i}$, the estimation of $W_{\gamma\tau}$ can be refined as follows:
\begin{lemma}\label{W}
    For any $(\gamma,\tau) \in \Gamma_\l$, we have
    $$W_{\gamma\tau}=\begin{cases}
        0 &(\kappa_{0} \ge 1)\\
        O_\l(q^{-r\kappa_{1}}) &(\kappa_{0}=0).
    \end{cases}$$
\end{lemma}

\section{Proof of Theorem \ref{M}} \label{high}

We remark that Theorem \ref{M} under the cases $\l=1,2$ are immediate by Lemma \ref{musigma}. Hence in the following we assume $\l \geq 3$.

Define $\Gamma_\l^0,\Gamma_\l^1,\Gamma_\l^2$ and $\Gamma_\l^3$ as follows:
\begin{align*}
    \Gamma_\l^0&:=\left\{(\gamma,\tau) \in \Gamma_\l: \kappa_{0} > 0\right\},\\
    \Gamma_\l^1&:=\left\{(\gamma,\tau) \in \Gamma_\l: \kappa_{0}=0, \kappa=\l/2\right\},\\
    \Gamma_\l^2&:=\left\{(\gamma,\tau) \in \Gamma_\l: \kappa_{0}=\kappa_{1}=0, \kappa < \l/2\right\},\\
    \Gamma_\l^3&:=\left\{(\gamma,\tau) \in \Gamma_\l: \kappa_{0}=0, \kappa_{1} > 0, \kappa < \l/2\right\}.
\end{align*}
Note that if $\kappa > \l/2$, then we must have $\kappa_{0} > 0$. Hence we see that $\Gamma_\l=\sqcup_{i=0}^3 \Gamma_\l^i$. For $i \in [0\isep 3]$, we define the sum
$$M_i:=\frac{1}{(\sigma_\A^2(q,m,n,r))^{\l/2}}\sum_{(\gamma,\tau) \in \Gamma_\l^i}m^{u_\gamma} n^{v_\tau} W_{\gamma\tau}.$$
Then by (\ref{E4}) and the fact that
$$\frac{m!}{(m-u)!}=m^u\left(1+O_u\left(\frac{1}{m}\right)\right),$$
we have
\begin{equation}\label{W1}
   \E[\cts(\mathbf{XY})^\l]=\sum_{i=0}^3 M_i\left(1+O_\l\left(\frac{1}{N}\right)\right).
\end{equation}
Here for the sake of simplicity we define $N:=\min\{m,n\}$. 

In what follows, we estimate $M_i$'s for each $i \in [0\isep 3]$. 

\subsection{$M_0$}
This is trivial: by Lemma \ref{W}, we immediately have $M_0=0$.

\subsection{$M_1$} 

For each $(\gamma,\tau) \in \Gamma_{\l}^1$, $\kappa=\l/2$ is a positive integer, so if $\l$ is odd, then $\Gamma_{\l}^1=\emptyset$ and we have $M_1=0$. 

Now let us assume that $\l$ is even. This means that each connected component $G_{\gamma_i\tau_i}$ has exactly two edges (counted with multiplicity),   which is either of type $T_1$ or of type $T_2$, according to whether it is a tree or a double edge. See pictures below: if it is of type $T_1$, it is either graph C1 or C2; if it is of type $T_2$, then it is graph D.

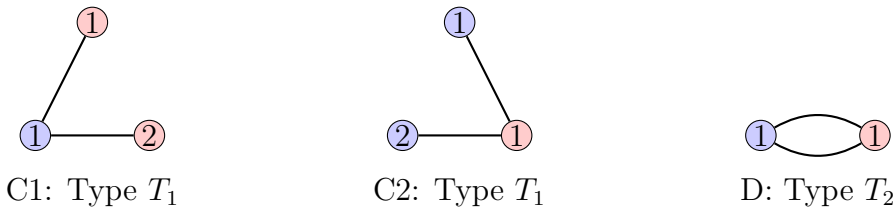
\begin{tikzpicture}[scale=1.5]

    \node[circle, draw, fill=blue!20, minimum size=0.4cm, inner sep=0pt] (C) at (0, 0) {};
    \node[circle, draw, fill=red!20, minimum size=0.4cm, inner sep=0pt] (D) at (1, 0) {};
    \node[circle, draw, fill=red!20, minimum size=0.4cm, inner sep=0pt] (E) at (0.5, 1) {};
    
    \draw[thick] (C) -- (D);
    \draw[thick] (C) -- (E);

    \node at (C) {1};
    \node at (D) {2};
    \node at (E) {1};

    \node at (0.5, -0.5) {C1: Type $T_1$};

\end{tikzpicture}
\hspace{2cm} 
\begin{tikzpicture}[scale=1.5]

    \node[circle, draw, fill=blue!20, minimum size=0.4cm, inner sep=0pt] (C) at (0, 0) {};
    \node[circle, draw, fill=red!20, minimum size=0.4cm, inner sep=0pt] (D) at (1, 0) {};
    \node[circle, draw, fill=blue!20, minimum size=0.4cm, inner sep=0pt] (E) at (0.5, 1) {};
    
    \draw[thick] (C) -- (D);
    \draw[thick] (D) -- (E);

    \node at (C) {2};
    \node at (D) {1};
    \node at (E) {1};

    \node at (0.5, -0.5) {C2: Type $T_1$};

\end{tikzpicture}
\hspace{2cm}
\begin{tikzpicture}[scale=1.5]

    \node[circle, draw, fill=blue!20, minimum size=0.4cm, inner sep=0pt] (A) at (0, 0) {};
    \node[circle, draw, fill=red!20, minimum size=0.4cm, inner sep=0pt] (B) at (1, 0) {};
    
    \draw[thick, bend left] (A) to (B);
    \draw[thick, bend right] (A) to (B); 

    \node at (A) {1};
    \node at (B) {1};

    \node at (0.5, -0.5) {D: Type $T_2$};

\end{tikzpicture}

Since all the vectors $\x_i,\y_j$ are uniformly and independently distributed in $\F_q^r$, it is easy to see that 
\begin{itemize}
\item if $G_{\gamma_i\tau_i} = \text{C1}$, then 
\[W_{\gamma_i\tau_i}=\E \Big[\big\{\mathbbm{1}_\A(\x_1\cdot\y_1)-\E[\mathbbm{1}_\A(\x_1\cdot\y_1)]\big\}\big\{\mathbbm{1}_\A(\x_1\cdot\y_2)-\E[\mathbbm{1}_\A(\x_1\cdot\y_2)]\big\}\Big]; \]

\item if $G_{\gamma_i\tau_i} = \text{C2}$, then 
\[W_{\gamma_i\tau_i}=\E \Big[\big\{\mathbbm{1}_\A(\x_1\cdot\y_1)-\E[\mathbbm{1}_\A(\x_1\cdot\y_1)]\big\}\big\{\mathbbm{1}_\A(\x_2\cdot\y_1)-\E[\mathbbm{1}_\A(\x_2\cdot\y_1)]\big\}\Big]; \]

\item if $G_{\gamma_i\tau_i} = \text{D}$, then 
\[W_{\gamma_i\tau_i}=\E \Big[\big\{\mathbbm{1}_\A(\x_1\cdot\y_1)-\E[\mathbbm{1}_\A(\x_1\cdot\y_1)]\big\}^2\Big]; \]

\end{itemize}
As we have seen in the proof of Lemma \ref{musigma}, we can obtain 
\begin{eqnarray} \label{5:wgamma} W_{\gamma_i\tau_i}=\left\{\begin{array}{ll}
q^{-r}A & (G_{\gamma_i\tau_i} \in \{ \text{C1, C2}\}),\\
B & (G_{\gamma_i\tau_i}= \text{D}), 
\end{array}
\right.\end{eqnarray}
where we define 
\begin{eqnarray*} 
A:&=& (1-q^{-r})\gaa^2, \\
B:&=&\left(q^{-1}|\A|-q^{-r}\gaa\right)\left(1-q^{-1}|\A|+q^{-r}\gaa\right). 
\end{eqnarray*}
It shall be noted that both $A$ and $B$ are of order 1 as $r \to +\infty$, and 
\begin{equation}\label{S}
\sigma_\A^2(q,m,n,r)=mn\left[(m+n-2)q^{-r}A+B\right].
\end{equation}
Suppose 
\begin{eqnarray} \label{5:decom} [1 \isep \kappa]=K_{11} \sqcup K_{12} \sqcup K_2 \end{eqnarray}
such that 
\[ \left\{
\begin{array}{ll} 
G_{\gamma_i\tau_i} =\text { C1 } & (i \in K_{11}), \\
G_{\gamma_i\tau_i} =\text { C2 } & (i \in K_{12}), \\
G_{\gamma_i\tau_i} =\text { D } & (i \in K_{2}). 
\end{array} \right.\] 
Denote 
\[ \#K_{11}=\kappa_{11}, \quad \#K_{12}=\kappa_{12}, \quad \#K_2=\kappa_2. \] 
These $\kappa$'s satisfy
$$\kappa_{11}+\kappa_{12}+\kappa_2=\kappa_1+\kappa_2=\l/2.$$
From (\ref{5:wgamma}) we obtain
$$W_{\gamma\tau}=\prod_{i} W_{\gamma_i \tau_i}=\left(q^{-r}A\right)^{\kappa_{11}+\kappa_{12}}B^{\kappa_{2}}. $$
Since 
\begin{itemize}
\item if $G_{\gamma_i\tau_i}$=C1, then $u_{\gamma_i}=1, v_{\tau_i}=2$, 

\item if $G_{\gamma_i\tau_i}$=C2, then $u_{\gamma_i}=2, v_{\tau_i}=1$, 

\item if $G_{\gamma_i\tau_i}$=D, then $u_{\gamma_i}=1, v_{\tau_i}=1$, 
\end{itemize}
we have
$$u_\gamma=\sum_{i=1}^\kappa u_{\gamma_i}=\kappa_{11}+2\kappa_{12}+\kappa_{2}=\l/2+\kappa_{12}$$
and
$$v_\tau=\sum_{i=1}^\kappa v_{\tau_i}=2\kappa_{11}+\kappa_{12}+\kappa_{2}=\l/2+\kappa_{11}.$$
Denote by $\Gamma_\l^1(K_{11},K_{12},K_2)$ the set of those $(\gamma,\tau) \in \Gamma_\l^1$ associated to the decomposition (\ref{5:decom}). A little thought reveals that the quantity $\#\Gamma_{\l}^1(K_{11},K_{12},K_2)$ counts exactly the number of ways to partition $[1\isep \l]$ into $\l/2$ disjoint two-element subsets (each of which corresponds to the indices of the two edges in a single connected component of $G_{\gamma\tau}$). So we have 
\begin{eqnarray} \label{5:gcount} \sum_{(\gamma,\tau) \in \Gamma_\l^{1}(K_{11},K_{12},K_2)}1=(\l-1)!!.\end{eqnarray}
Now using the decomposition  
\[
\Gamma_{\l}^1=\bigsqcup_{K_{11}, K_{12},K_2} \Gamma_\l^1(K_{11},K_{12},K_2), \]
we have 
\begin{align*}
    M_1&=\frac{1}{(\sigma_\A^2(q,m,n,r))^{\l/2}}\sum_{K_{11},K_{12},K_2} \, \sum_{(\gamma,\tau) \in \Gamma_{\l}^1(K_{11},K_{12},K_2)} m^{u_{\gamma}}n^{v_{\tau}}W_{\gamma \tau} \\
	&=\frac{1}{(\sigma_\A^2(q,m,n,r))^{\l/2}}\sum_{K_{11},K_{12},K_2} \, \sum_{(\gamma,\tau) \in \Gamma_{\l}^1(K_{11},K_{12},K_2)} m^{\kappa_{11}+2\kappa_{12}+\kappa_{2}}n^{2\kappa_{11}+\kappa_{12}+\kappa_{2}}\left(q^{-r}A\right)^{\kappa_{11}+\kappa_{12}}B^{\kappa_{2}}\\			&=\frac{1}{(\sigma_\A^2(q,m,n,r))^{\l/2}}\sum_{K_{11},K_{12},K_2} \, \sum_{(\gamma,\tau) \in \Gamma_{\l}^1(K_{11},K_{12},K_2)} m^{\l/2+\kappa_{12}}n^{\l/2+\kappa_{11}}\left(q^{-r}A\right)^{\kappa_{11}+\kappa_{12}}B^{\l/2-\kappa_{11}-\kappa_{12}}\\
	&=\frac{(mn)^{\l/2}}{(\sigma_\A^2(q,m,n,r))^{\l/2}}\sum_{\kappa_{11},\kappa_{12}} m^{\kappa_{12}}n^{\kappa_{11}}\left(q^{-r}A\right)^{\kappa_{11}+\kappa_{12}}B^{\l/2-\kappa_{11}-\kappa_{12}} \sum_{\substack{K_{11},K_{12},K_2\\
	\#K_{11}=\kappa_{11}\\
	\#K_{12}=\kappa_{12}}} \, \sum_{(\gamma,\tau) \in \Gamma_{\l}^1(K_{11},K_{12},K_2)}1. 
\end{align*}
By using Identity (\ref{5:gcount}) we have 
\begin{align*}
    M_1&=\frac{(mn)^{\l/2}(\l-1)!!}{(\sigma_\A^2(q,m,n,r))^{\l/2}}\sum_{\kappa_{11},\kappa_{12}} m^{\kappa_{12}}n^{\kappa_{11}}\left(q^{-r}A\right)^{\kappa_{11}+\kappa_{12}}B^{\l/2-\kappa_{11}-\kappa_{12}} \sum_{\substack{K_{11},K_{12},K_2\\
	\#K_{11}=\kappa_{11}\\
	\#K_{12}=\kappa_{12}}} 1.
	\end{align*}
Noting that subject to Condition (\ref{5:decom}) on $K_{11},K_{12}$ and $K_2$, we have 
	\[\sum_{\substack{K_{11},K_{12},K_2\\
	\#K_{11}=\kappa_{11}\\
	\#K_{12}=\kappa_{12}}} 1=\binom{\l/2}{\kappa_{11}+\kappa_{12}} \binom{\kappa_{11}+\kappa_{12}}{\kappa_{11}}, \]
we can now compute $M_1$ by simple applications of the binomial theorem $(a+b)^n=\sum_{i=0}^n \binom{n}{i} a^ib^{n-i}$: setting $\kappa_1=\kappa_{11}+\kappa_{12}$ so that $\kappa_{12}=\kappa_1-\kappa_{11}$, noting that $0 \le \kappa_{11} \le \kappa_1 \le \l/2$ and change the order of summation, we can further obtain 
\begin{align*}
    M_1&=\frac{(mn)^{\l/2}(\l-1)!!}{(\sigma_\A^2(q,m,n,r))^{\l/2}}\sum_{\kappa_{1}=0}^{\l/2} \binom{\l/2}{\kappa_1}\left(q^{-r}A\right)^{\kappa_{1}}B^{\l/2-\kappa_{1}}\sum_{\kappa_{11}=0}^{\kappa_1} \binom{\kappa_1}{\kappa_{11}} m^{\kappa_1-\kappa_{11}}n^{\kappa_{11}} \\
	&=\frac{(mn)^{\l/2}(\l-1)!!}{(\sigma_\A^2(q,m,n,r))^{\l/2}}\sum_{\kappa_{1}=0}^{\l/2} \binom{\l/2}{\kappa_1} \left((m+n)q^{-r}A\right)^{\kappa_{1}}B^{\l/2-\kappa_{1}}\\
    &=\frac{(\l-1)!!\left\{mn[(m+n)q^{-r}A+B]\right\}^{\l/2}}{(\sigma_\A^2(q,m,n,r))^{\l/2}}. 
\end{align*}
Using the value of $\sigma_\A^2(q,m,n,r)$ given in (\ref{S}), it is straightforward to obtain 

$$M_1 =(\l-1)!!+O_\l\left(\frac{1}{m+n}\right).$$ 

\subsection{$M_2$} 
Decompose $G_{\gamma \tau}$ into connected components
\[G_{\gamma \tau} =\sqcup_{i=1}^{\kappa} G_{\gamma_i \tau_i},\]
where $G_{\gamma_i \tau_i}=(U_{\gamma_i},V_{\tau_i}, E_{\gamma_i \tau_i})$ is the graph associated to $(\gamma_i, \tau_i )\in \Gamma_{\l_i}$ for $1 \le i \le \kappa$. 

If $(\gamma,\tau) \in \Gamma_{\l}^2$, then $\kappa <\l/2$ and each $G_{\gamma_i \tau_i}$ is of type $T_2$, that is, each edge of $G_{\gamma_i\tau_i}$ is a multi-edge. We first see that $W_{\gamma_i\tau_i}=O_\l(1)$ for each $i$ by Lemma \ref{W0}. Next, since each component only has multiple edges and is connected, we have $\ep_{\gamma_i\tau_i}' \leq \l_{i}/2$, and hence $u_{\gamma_i}+v_{\tau_i} \leq \ep_{\gamma_i\tau_i}+1 \leq  \l_{i}/2+1$ for all $i$. As $\sum_i \l_i=\l$, this implies that 
$$u_{\gamma}+v_{\tau} =\sum_i \left(u_{\gamma_i}+v_{\tau_i} \right)\leq \l/2+\kappa.$$
So we have 
\begin{align*}
    M_2&=\frac{1}{(\sigma_\A^2(q,m,n,r))^{\l/2}}\sum_{(\gamma,\tau) \in \Gamma_{\l}^2} m^{u_{\gamma}} n^{v_{\tau}}\prod_{i=1}^{\kappa}W_{\gamma_i \tau_i} 
    \\
&\ll_\l\frac{1}{(\sigma_\A^2(q,m,n,r))^{\l/2}}\sum_{\kappa=1}^{\lfloor (\l-1)/2\rfloor}\sum_{u}\sum_{v}m^u n^{v}\sum_{\substack{(\gamma,\tau) \in \Gamma_\l^2 \\ u_\gamma=u \\ v_{\tau}=v}} 1\\
&\ll_\l\frac{1}{(\sigma_\A^2(q,m,n,r))^{\l/2}}\sum_{\kappa=1}^{\lfloor (\l-1)/2\rfloor}\sum_{u}\sum_{v}m^u n^{v}.
\end{align*}
Here the constraints on $u,v$ are $u,v \ge \kappa$ and $u+v \le \l/2+\kappa$. We can obtain 
\begin{align*}
    M_2    &\ll_\l \frac{1}{(\sigma_\A^2(q,m,n,r))^{\l/2}}\sum_{\kappa=1}^{\lfloor (\l-1)/2\rfloor}(mn)^\kappa(m+n)^{\lfloor\l/2\rfloor-\kappa}\\
    &\ll_\l \frac{(m+n)^{\lfloor\l/2\rfloor}}{(\sigma_\A^2(q,m,n,r))^{\l/2}}\sum_{\kappa=1}^{\lfloor (\l-1)/2\rfloor} \left(\frac{mn}{m+n}\right)^\kappa\\
    &\ll_\l \frac{(m+n)^{\lfloor\l/2\rfloor}N^{\lfloor (\l-1)/2\rfloor}}{(mn)^{\l/2}} \ll_\l \frac{1}{N}.
\end{align*}

\subsection{$M_3$} 

Decompose $G_{\gamma \tau}$ into connected components
\[G_{\gamma \tau} =\sqcup_{i=1}^{\kappa} G_{\gamma_i \tau_i},\]
where $G_{\gamma_i \tau_i}=(U_{\gamma_i},V_{\tau_i}, E_{\gamma_i \tau_i})$ is the graph associated to $(\gamma_i, \tau_i )\in \Gamma_{\l_i}$ for $1 \le i \le \kappa$. 

For $(\gamma,\tau) \in \Gamma_{\l}^3$, each component $G_{\gamma_i\tau_i}$ is either of type $T_1$ or of type $T_2$. Let 
\[[1 \isep \kappa]=K_1 \sqcup K_2,\]
where 
\[G_{\gamma_i \tau_i} \text{ is } \left\{
\begin{array}{ll} 
\text{ of type } T_1 & (i \in K_1),\\ 
\text{ of type } T_2 & (i \in K_2),
\end{array}\right.\]
and 
$$\kappa_1=\#K_1>0, \quad \kappa_2=\#K_2, \quad \kappa_1+\kappa_2 =\kappa<\l/2.$$ 
Since each $G_{\gamma_i\tau_i}$ is connected, we have 
\begin{itemize}
\item if $i \in K_1$, then $u_{\gamma_i}+v_{\tau_i} \le \l_i+1$, 

\item if $i \in K_2$, then $u_{\gamma_i}+v_{\tau_i} \le \l_i/2+1$. 
\end{itemize}
Define $$\l_1:=\sum_{i \in K_1}\l_{i}, \quad \l_2:=\sum_{i \in K_2}\l_{i}.$$ 
We have
\[\l_1+\l_2=\l, \quad \l_1 \ge 2 \kappa_1, \quad \l_2 \ge 2 \kappa_2.\]
Then $2\kappa_{1} \leq \l_1 \leq \l-2\kappa_{2}$ since $\kappa_{\gamma\tau,0}=0$. In addition,
\begin{eqnarray} \label{5:m3} u_\gamma \geq \kappa, v_\tau \geq \kappa, u_\gamma+v_\tau =\sum_i\left(u_{\gamma_i}+v_{\tau_i}\right)\leq \l_1+\l_2/2+\kappa=(\l_1+\l)/2+\kappa.\end{eqnarray}
We can write 
\begin{align}
    M_3&=\frac{1}{(\sigma_\A^2(q,m,n,r))^{\l/2}} \sum_{(\gamma,\tau) \in \Gamma_{\l}^3} m^{u_{\gamma}}n^{v_{\tau}} \prod_{i=1}^{\kappa} W_{\gamma_i\tau_i} \nonumber\\
    &\ll_\l\frac{1}{(\sigma_\A^2(q,m,n,r))^{\l/2}}\sum_{\kappa=1}^{\lfloor (\l-1)/2\rfloor} \sum_{u_{\gamma},v_{\tau},\kappa_1}m^{u_{\gamma}} n^{v_{\tau}}q^{-r\kappa_1} \sum_{\substack{(\gamma,\tau) \in \Gamma_\l^3 \\ u_{\gamma}=u\\
    v_\tau=v\\
    \#K_1=\kappa_1}} 1\nonumber\\
    &\ll_\l\frac{1}{(\sigma_\A^2(q,m,n,r))^{\l/2}}\sum_{\kappa=1}^{\lfloor (\l-1)/2\rfloor} (mn)^{\kappa}\sum_{u_{\gamma},v_{\tau},\kappa_1}m^{u_{\gamma}-\kappa} n^{v_{\tau}-\kappa}q^{-r\kappa_1}. \nonumber
\end{align}
The summation above is under the constraints for $u_{\gamma},v_{\tau}$ appearing in (\ref{5:m3}) and $1 \leq \kappa_1 \leq \kappa$. We then obtain 
\begin{align} \label{M3}
    M_3
    &\ll_\l \frac{1}{(\sigma_\A^2(q,m,n,r))^{\l/2}}\sum_{\kappa=1}^{\lfloor (\l-1)/2\rfloor} (mn)^\kappa \sum_{\kappa_1=1}^\kappa\sum_{\l_1=2\kappa_1}^{\l-2\kappa+2\kappa_1}(m+n)^{\lfloor(\l_1+\l)/2\rfloor-\kappa}q^{-r\kappa_1}\nonumber\\
    &\ll_\l \frac{(m+n)^\l}{(\sigma_\A^2(q,m,n,r))^{\l/2}}\sum_{\kappa=1}^{\lfloor (\l-1)/2\rfloor}\left[\frac{mn}{(m+n)^2}\right]^\kappa\sum_{\kappa_1=1}^\kappa [(m+n)q^{-r}]^{\kappa_1}.
\end{align}
\noindent {\bf Case 1.} Suppose  
$$\lim_{m,n \to \infty} \frac{q^r}{m+n}=0.$$ 
By (\ref{S}), we have 
$$\sigma_\A^2(q,m,n,r) \asymp mn(m+n)q^{-r}.$$ 
Thus (\ref{M3}) yields
\begin{align*}
    M_3&\ll_\l \frac{(m+n)^\l}{[mn(m+n)q^{-r}]^{\l/2}}\sum_{\kappa=1}^{\lfloor (\l-1)/2\rfloor}\left[\frac{mn}{(m+n)^2}\right]^\kappa[(m+n)q^{-r}]^{\kappa}\\
    &\ll_\l \left(\frac{q^r}{N}\right)^{\l/2}\sum_{\kappa=1}^{\lfloor (\l-1)/2\rfloor} \left(\frac{N}{q^r}\right)^{\kappa}.
\end{align*}
Here $N:=\min\{m,n\}$. If we assume further that 
\begin{eqnarray} \label{5:aN}
\lim_{m,n \to \infty} \frac{q^r}{N}=0, \end{eqnarray}
then the above implies
$$M_3 \ll_\l \begin{cases}
    \sqrt{\frac{q^r}{N}} &(\l \text{ odd})\\
    \frac{q^r}{N} &(\l \text{ even}).
\end{cases}$$

\noindent {\bf Case 2.} Suppose 
$$\lim_{m,n \to \infty} \frac{q^r}{m+n}=\infty.$$ 
By (\ref{S}), we have $\sigma_\A^2 \asymp mn$. Hence (\ref{M3}) yields
\begin{align*}
    M_3&\ll_\l \frac{(m+n)^\l}{(mn)^{\l/2}}\sum_{\kappa=1}^{\lfloor (\l-1)/2\rfloor}\left[\frac{mn}{(m+n)^2}\right]^\kappa(m+n)q^{-r}\\
    &\ll_\l \left[\frac{(m+n)^2}{mn}\right]^{\l/2-1}(m+n)q^{-r}, 
\end{align*}
since $0<\frac{mn}{(m+n)^2}<\frac{1}{2}$. If we assume further that 
\begin{eqnarray} \label{5:aN2}  
\lim_{m,n \to \infty} \frac{q^r}{(m+n)^a}=\infty, \qquad \text{ for any fixed }a>0, 
\end{eqnarray}
or
\begin{eqnarray} \label{5:aN3}  
m \asymp n, \quad \text{ and } \quad \lim_{m\to \infty} \frac{q^r}{m}=\infty, 
\end{eqnarray}
then we can still conclude that $M_3=o_{\l}(1)$.

Putting all above estimates of $M_0,M_1,M_2,M_3$ into (\ref{W1}), we conclude that for any $\l \geq 3$,
\begin{itemize}
    \item[(1)] Under Assumption (\ref{5:aN}), 
$$\E[\cts(\mathbf{XY})^\l]=\begin{cases}
    O_\l\left(\sqrt{\frac{q^r}{N}}\right) &(\l \text{ odd})\\
    (\l-1)!!+O_\l\left(\frac{q^r}{N}\right) &(\l \text{ even});
\end{cases}$$
\item[(2)] Under Assumptions (\ref{5:aN2}) or (\ref{5:aN3}), 
$$\E[\cts(\mathbf{XY})^\l]=\begin{cases}
    o_\l\left(1\right) &(\l \text{ odd})\\
    (\l-1)!!+o_\l\left(1\right) &(\l \text{ even}).
\end{cases}$$
\end{itemize}
Here $N=\min\{m,n\}$. This completes the proof of Theorem \ref{M}. \hfill $\square$

\bibliographystyle{elsarticle-num}

\bibliography{CX-2024}
\end{document}